\documentclass{amsart}

\usepackage{amsmath,tabu}
\usepackage{amsfonts}
\usepackage{amssymb}
\usepackage{amsthm}
\usepackage[bookmarks]{hyperref}
\hypersetup{colorlinks=false}
\usepackage{cleveref}
\usepackage{setspace}
\usepackage{enumerate}
\usepackage{amscd}
\usepackage{tikz-cd}

\newcommand{\mS}{\mathcal{S}}
\newcommand{\mT}{\mathcal{T}}

\newcommand{\ot}{\otimes}
\newcommand{\ra}{\rightarrow}
\newcommand{\C}{\mathbb{C}}

\newcommand{\B}{B(\mathcal{H})}

\newcommand{\mH}{\mathcal{H}}

\newcommand{\N}{\mathbb{N}}

\newcommand{\mO}{\mathcal{O}}

\newcommand{\mF}{\mathfrak{u}}

\theoremstyle{thmit}
\newtheorem{theorem}{Theorem}[section]

\newtheorem{proposition}[theorem]{Proposition}
\newtheorem{lemma}[theorem]{Lemma}
\newtheorem{corollary}[theorem]{Corollary}

\theoremstyle{thmrm}
\newtheorem{example}[theorem]{Example}
\newtheorem{remarks}[theorem]{Remark}
\title[Embeddings of exact operator systems]{\textbf{Embeddings and $C^*$-envelopes of exact operator systems}}
\date{\today}

\author[Preeti Luthra]{Preeti Luthra}
\address{Department of Mathematics\\ University of Delhi\\ Delhi-110007, INDIA}
\email{maths.preeti@gmail.com}
\author[Ajay Kumar]{Ajay Kumar}
\address{Department of Mathematics\\ University of Delhi\\ Delhi-110007, INDIA}
\email{akumar@maths.du.ac.in}


\begin{document}
\maketitle

\begin{abstract}
We prove a necessary and sufficient condition for embeddability of an operator system into $\mO_2$. Using Kirchberg's theorems on a tensor product of $\mO_2$ and $\mO_{\infty}$, we establish results on their operator system counterparts $\mS_2$ and $\mS_{\infty}$. Applications of the results proved, including some examples describing $C^*$-envelopes of operator systems, are also discussed.
\end{abstract}

\section{Introduction}\label{s:1}
The study of operator systems with universal generators generating some well studied $C^*$-envelopes has attracted considerable interest in recent years (see \cite{disgrp, opsysqt, ortiz, zheng}). Da Zheng (\cite{zheng}) introduced the operator system $\mS_n$ generated by Cuntz isometries and later, in \cite{zheng2}, Paulsen and Zheng explored the tensor product and nuclearity related properties of this operator system.
\\
Cuntz introduced the $C^*$-algebras $\mO_n$($1 \leq n \leq \infty$)(see \cite{cuntz}), in the year 1977, which were the first explicit examples of simple infinite separable $C^*$-algebras. Cuntz proved that his algebras are simple and purely infinite, and are independent of the choice of generators.
\\
In fact, these algebras played an important role in the classification theory of purely infinite, simple, separable and nuclear $C^*$-algebras, by Kirchberg and Philips. The classification theory for separable $C^*$-algebras with certain properties in terms of Cuntz algebras $\mO_2$ and $\mO_{\infty}$ was given by Kirchberg and Rordam. One can refer to \cite{rordam} for a detailed discussion on this classification programme.
\\
There are basically three fundamental theorems given by Kirchberg; namely, the embedding of separable exact $C^*$-algebras into Cuntz algebra $\mO_2$, and the tensor product theorems for $\mO_2$ and $\mO_{\infty}$.
Many more generalizations of these results were later proved by Kirchberg and Rordam. In a recent work, Lupini \cite{lupini2}  has established an operator system analog of Kirchberg's nuclear embedding theorem involving the Gurarij operator system $\mathbb{GS}$.
\\
Since, for all $ 1\leq n \leq \infty$, $\mO_n$ is a simple $C^*$-algebra, it turns out that $\mO_n$ is in fact the $C^*$-envelope of $\mS_n$(\cite{zheng}). This motivates us to study the Kirchberg's theorems on $\mO_n$ ($2 \leq n \leq \infty$) in terms of the $C^*$-envelopes of operator system.
\\
After collecting prerequisites in \Cref{s:2}, we prove an embedding theorem for operator systems motivated by Kirchberg's exact embedding theorem in \Cref{s:3}. It gives a necessary and sufficient condition for embedding an operator system into $\mO_2$ in terms of exactness of its $C^*$-envelopes. Further, we could extend these embeddibility conditions to finite \emph{minimal} tensor products of operator systems. We also discuss some nuclearity properties of operator systems embedding in $\mO_2$.
\\
In \Cref{s:4}, results regarding embedding of operator systems of the form $\mS \ot_{\mathrm{min=c}} \mS_2$ into $\mO_2$ are proved. Further, we obtain some equivalent conditions for their $C^*$-envelopes to be $*$-isomorphic either to $\mO_2$ or to a $C^*$-subalgebra of $\mO_2$. We also prove results on operator system of the form $\mS \ot_{\mathrm{min=c}} \mS_{\infty}$ using Kirchberg's theorems on a tensor product of $\mO_{\infty}$.
\\
Finally in \Cref{s:5}, as an application of results proved, we check the embeddability of some operator systems whose $C^*$-envelopes are already calculated, into $\mO_2$. A description of $C^*$-envelopes of some operator systems with tensor product  factor $\mS_2$ or $\mS_{\infty}$ is also given, which adds few more operator systems to the short list of known $C^*$-envelopes.


\section{Preliminaries}\label{s:2}
\subsection{Cuntz algebra and Kirchberg's theorems} 

The \emph{Cuntz algebra \cite{cuntz} $\mathcal{O}_n$}, where $2 \leq n < \infty$, is the universal unital $C^*$-algebra generated by isometries $s_1,s_2,\ldots, s_n$ satisfying $s_1s_1^*+s_2s_2^*+\ldots+ s_ns_n^*=1$. The Cuntz algebra $\mathcal{O}_{\infty}$ is the universal unital $C^*$-algebra generated by an infinite sequence of isometries $s_1,s_2,s_3,\ldots$ with mutually orthogonal range projections $s_js_j^*$ which add up to identity.
\\
A finite set $\{t_j\}_{j=1}^{ n}$ of isometries in a unital $C^*$-algebra $A$ is said to satisfy the \emph{Cuntz relation} if $t_1t_1^*+t_2t_2^*+ \ldots + t_n t_n^*=1$, and a sequence $\{t_j\}_{j=1}^{\infty}$ of isometries satisfies the Cuntz relation if their range projections $\{t_j t_j^*\}_{j=1}^{\infty}$ are mutually orthogonal. Cuntz algebra are well defined in the sense that they are independent of the choice of generating isometries.
\\
A self-contained survey of classical theorems stated below can be found in \cite{rordam}.
\\\\
The class of $C^*$-algebras $\mO_n$ have the following properties, as proved in \cite{cuntz}: 

\begin{theorem}\label{cuntz} For each $n \in \mathbb{N}$ and for $n = \infty$, the Cuntz algebra $\mO_n$ is unital, separable,
simple, nuclear and purely infinite.
\end{theorem}

Elliot proved an important self-absorbing property of $\mO_2$ in 1993 :

\begin{theorem}\label{elliot} The $C^*$-algebras
$\mathcal{O}_2 \ot_{{C^*\text{-}\min}} \mathcal{O}_2$ and $\mathcal{O}_2$ are isomorphic.  
\end{theorem}

In the year 2000, Kirchberg (\cite[Theorem 2.8]{kirchbergcuntz}) characterized separable exact $C^*$-algebras through the following result:

\begin{theorem}\label{kirchexact}
A unital separable $C^*$-algebra $A$ is exact if and only if it admits a unital embedding
into $\mO_2$.
\end{theorem}

A complete characterization of $C^*$-algebras $A$ for which ${A} \ot_{C^*\text{-}\min} \mathcal{O}_2 \cong \mathcal{O}_2$ was also given by  Kirchberg(\cite[Theorem 3.7]{kirchbergcuntz}).

\begin{theorem}\label{kirch2}
The tensor product $A \ot_{C^*\text{-}\min} \mathcal{O}_2$ is isomorphic to $\mathcal{O}_2$ if and only if A is unital, simple,
separable and nuclear.
\end{theorem}

Kirchberg \cite[Theorem 7.2.6]{kirchbergcuntz} also gave a classification of all separable, nuclear $C^*$-algebras that absorb $\mO_{\infty}$:

\begin{theorem}\label{Oinfinity} For a simple, nuclear, separable $C^*$-algebra $A$, $A \cong A \ot_{C^*\text{-}\min} \mathcal{O}_{\infty} $ if and only if $A$ is purely infinite.
\end{theorem}

We'll also use the following permanence properties several times ahead (\cite[Theorem 6.1.10]{rordam}\cite[Corollary 4.21]{takesaki}):

\begin{theorem}\label{propexact}
\begin{enumerate}
\item[(i)] Every $C^*$-subalgebra of an exact $C^*$-algebra is again exact.
\item[(ii)] Every quotient of an exact $C^*$-algebra is again exact.
\item[(iii)] If $A$ and $B$ are exact then so is $A \ot_{C^*\text{-}\min} B$.
\item[(iv)] If $A$ and $B$ are simple $C^*$-algebras then $A \ot_{C^*\text{-}\min} B$ is also simple.
\end{enumerate}
\end{theorem}


\subsection{Operator systems}
The concept of operator systems and their tensor products is the familiar one now and most of the details can be seen in \cite{kavnuc, KPTT1, KPTT2}.
\\
Recall that a concrete operator system is a unital self-adjoint subspace of $B(H)$ for some Hilbert space $H$. A \textit{$C^*$-cover} (\cite[$\S 2$]{hamana2}) of an operator system
$\mS$ is a pair $(A, i)$ consisting of a unital $C^*$-algebra $A$ and
a complete order embedding $i : \mS \rightarrow A$ such that $i(A)$
generates the $C^*$-algebra $A$. The \textit{$C^*$-envelope} as defined by Hamana \cite{hamana2}, of an operator system $\mS$ is a
$C^*$-cover defined as the $C^*$-algebra generated by $\mS$ in its
injective envelope $I(\mS)$ and is denoted by $C^*_e (\mS)$.  The $C^*$-envelope $C^*_e (\mS)$
enjoys the following universal ``minimality'' property (\cite[Corollary 4.2]{hamana2}):

{\em Identifying $\mS$ with its image in $C^*_e (\mS)$, for any $C^*$-cover
$(A, i)$ of $\mS$, there is a unique surjective unital
$*$-homomorphism $\pi : A \ra C^*_e (\mS)$ such that $\pi(i(s)) = s$
for every $s$ in $\mS$.}

\begin{remarks}\label{simplecover} If an operator system $\mS$ has a simple $C^*$-cover $(A,i)$ then using this minimality property $\pi$ is injective and $A$ is $*$-isomorphic to the $C^*$-envelope of $\mS$.
\end{remarks}

From \cite{zheng}, for the generators $s_1,s_2,\ldots, s_n$ ($n \geq 2$) of the Cuntz algebra $\mO_n$ and  identity $I$,  the Cuntz operator system $\mS_n$ denotes the operator system generated by $s_1,s_2,\ldots, s_n$, that is,
$$\mS_n = span\{I, s_1,s_2,\ldots, s_n, s_1^*,s_2^*,\ldots, s_n^*\} \subset \mO_n.$$
 Similarly, for the generators  $s_1,s_2,\ldots$ of
$\mO_{\infty}$,$$ \mS_{\infty} = span\{I, s_1,s_2,\ldots, s_1^*,s_2^*,\ldots \} \subset \mO_{\infty}.$$

The following well known fact follows directly from \Cref{simplecover} and \Cref{cuntz}:

\begin{proposition}\cite{zheng}\label{cntznvlp}
 $C^*_e(\mS_n) = \mO_n$ for all $1 \leq n \leq \infty.$
\end{proposition}

A lattice of tensor products of
operator systems admitting a natural partial order: 
 $$min \leq e \leq el , er \leq  c \leq max,$$were introduced in \cite{KPTT1}. In \cite{disgrp}, a natural operator system
tensor product ``ess'' arising from the enveloping $C^*$-algebras,
viz., $\mathrm{\mS \ot_{ess} \mT \subseteq C^*_e(\mS) \ot_{max} C^*_e(\mT)},$
was also defined.
\\\\
The notion of \emph{exactness} saw its relevance in the theory of
operator systems after Kavruk et al.(2013) appropriately
formalized the notion of quotient for operator systems. 
\\
An operator system $\mS$ is said to be 
 $\mathrm{exact}$ if 
for every unital $C^*$-algebra $A$ and a closed ideal $I$ in $A$ sequence
$$0 \longrightarrow  \mS \hat{\ot}_{\mathrm{min}} I \longrightarrow \mS \hat{\ot}_{\mathrm{min}} A \ra \mS \hat{\ot}_{\mathrm{min}}
(A/ I) \ra 0$$ is exact.
\\\\
Given two operator system tensor
products $\alpha$ and $\beta$, an operator system $\mS$ is said to be
\emph{$(\alpha, \beta)$-nuclear} if the identity map between $\mS
\ot_{\alpha} \mT$ and $\mS \ot_{\beta} \mT$ is a complete order
isomorphism for every operator system $\mT$, i.e. $$\mS \ot_{\alpha}
\mT = \mS \ot_{\beta} \mT.$$ Also, an operator system $\mS$ is said to
be \textit{$C^*$-nuclear}, if $$\mS \ot_{\mathrm{min}} A =\mS
\ot_{\mathrm{max}} A$$ for all unital $C^*$-algebras $A$. For a $C^*$-algebra $A$, $A \ot_{c} \mS=A \ot_{\max} \mS$ for every operator system $\mS$ (\cite[Theorem $6.7$]{KPTT1}), that is, $A$ is $(\min,\mathrm{max})$-nuclear if and only if it is $(\min, \mathrm{c})$-nuclear.\\Exactness is one of the few intrinsic properties of operator systems
that has been used as a tool, by Kavruk et al.~(see for example
\cite{kavnuc}), in characterizing nuclearity properties of operator
systems.

\begin{theorem}\cite[Proposition 1.1, Corollary 2.8]{zheng2}\label{cuntzopsysnuc}
$\mS_n$ is $(\min,\mathrm{c})$-nuclear but not $(\mathrm{min},\mathrm{max})$-nuclear.
\end{theorem}

\begin{remarks} \emph{Since $C^*_e(\mS_n)=\mO_n$ is $C^*$-nuclear (\Cref{cuntz}), using \cite[Proposition $4.2$]{opsysnuc} one can conclude that $\mS_n$ is $(\mathrm{min},\mathrm{ess})$-nuclear. Further by \cite[Proposition $5.2$]{opsysnuc}, for $1 \leq n < \infty$, $\mS_n$ is not $(\mathrm{ess},\mathrm{max})$-nuclear, giving an alternate proof of the fact that $\mS_n$ is not $(\mathrm{min},\mathrm{max})$-nuclear.}
\end{remarks}

One more
fundamental $C^*$-cover, the maximal one, is associated to an operator system $\mS$,
namely, the universal $C^*$-algebra $C^*_u(\mS)$ introduced by
Kirchberg and Wassermann (\cite[$\S 3$]{kirchberg}). $C^*_u(\mS)$ satisfies the following universal ``maximality'' property:
\\
\emph{Every unital completely positive map $\phi: \mS \ra A$, where $A$ is a unital $C^*$-algebra extends uniquely to a unital $*$-homomorphism $\pi : C^*_u(\mS) \ra A$.}
\\
Recall from \cite{KPTT2}, a subspace $J$ of an operator system $\mS$ is said to be a \textit{kernel} if it is the kernel of some unital completely positive map from $\mS$ into some operator system $\mT$.
It was shown in \cite[Corollary $3.8$]{KPTT2} that a subspace $J $ of an operator system $\mS$ is a kernel of $\mS$ if and only if $J$ is an intersection of a closed two-sided ideal in $C^*_u(\mS)$ with $\mS$. This does create a curiosity regarding simplicity of $C^*_u(\mS)$. But the following result shows that $C^*_u(\mS)$ is never simple.

\begin{proposition}\label{:p} 
For an operator system $\mS$ with $dim(\mS) > 1$; $C^*_u(\mS)$ is not simple.
\end{proposition}
\begin{proof}
 Let $J \subset \mS$ be kernel in $\mS$, then by \cite[Corollary 3.8]{KPTT2}, $J= I \cap \mS$ for some closed two-sided ideal $I$ in $C^*_u(\mS)$. If $C^*_u(\mS)$ is simple then either $J=(0)$ or $J=\mS$. But by \cite[Corollary 6.12]{kavnuc}, any operator system with dimension greater than 1 has a non-trivial kernel, a contradiction. 
\end{proof}

Using minimality property of $C^*$-envelopes, there is a surjective $*$-
homomorphism $\sigma_{\mS} : C^*_u (\mS) \rightarrow C^*_e (\mS)$ that fixes $\mS$. And, hence simplicity of $C^*_u(S)$ implies simplicity of $C^*_e(\mS)$ (\Cref{simplecover}).
Therefore, an operator system kernel has no relation with the simplicity of its $C^*$-envelope.
\\
An operator system $\mS$ for which $\sigma_{\mS}$ is a $*$-isomorphism is
said to be \emph{universal} (\cite{kirchberg}). In particular, this property implies that if $\sigma_S : \mS \rightarrow A$ is any $C^*$-cover
of $S$, then $A \cong C^*_u (\mS) \cong C^*_e (\mS)$. Thus preceding proposition implies the following.

\begin{corollary}
 There does not exist any universal operator system $\mS$ with simple $C^*$-cover unless $\mS =\mathbb{C}$.
\end{corollary}

In general, the isomorphism between operator systems need not extend to their $C^*$-covers; but the following result from \cite{arveson} is quite useful.

\begin{theorem}\cite[Theorem $2.2.5$]{arveson}\label{extendiso}
For $\mS \subseteq C^*_e(\mS)$ and $\mT \subseteq C^*_e(\mT)$, for any complete order isomorphism $\phi$ of $ \mS$ onto $\mT$, there exists a $*$-isomorphism $\hat{\phi}$ from $C^*_e(\mS)$ onto $C^*_e(\mT)$, with $\hat{\phi}|_{\mS}= \phi.$
\end{theorem}

Recall an operator subsystem $\mS$ of a unital $C^*$-algebra $A$ is said \textit{to
contain enough unitaries of $A$} if the unitaries in $\mS$ generate $A$
as a $C^*$-algebra (\cite[$\S 9$]{KPTT2}).
Next lemma, although a folklore, is used several times.

\begin{lemma}\label{lemma} 
For operator systems $\mS$ and $\mT$ with either both  $C^*_e(\mS)$ and $C^*_e(\mT)$ simple or both $\mS$ and $\mT$ having enough unitaries of $C^*_e(\mS)$ and $C^*_e(\mT)$, respectively, the inclusion of $\mS \ot_{\min} \mT$ into $C^*_e(\mS) \ot_{C^*\text{-}{\min}} C^*_e(\mT)$  extends to a $*$-isomorphism between $C^*_e(\mS \ot_{\min} \mT)$ and $C^*_e(\mS) \ot_{C^*\text{-}{\min}} C^*_e(\mT)$, that is,
$$C^*_e(\mS \ot_{\min} \mT) \cong C^*_e(\mS) \ot_{C^*\text{-}{\min}} C^*_e(\mT).$$
\end{lemma}
\begin{proof}
Consider the natural inclusions $i_{\mS} : \mS \hookrightarrow C^*_e(\mS)$ and $i_{\mT} : \mT \hookrightarrow C^*_e(\mT)$. Then $i_{\mS} \ot i_{\mT} : \mS \ot_{\min} \mT \hookrightarrow C^*_e(\mS) \ot_{C^*\text{-}{\min}} C^*_e(\mT)$, is a $C^*$-cover of $\mS \ot_{\min} \mT$, and in fact it is simple by \Cref{propexact}(iv). Hence by \Cref{simplecover} statement follows. 
\\
For the enough unitaries case just note the fact that  $\mS \ot_{\min} \mT$  has enough unitaries of $C^*_e(\mS) \ot_{C^*\text{-}{\min}} C^*_e(\mT)$ and apply \cite[Proposition 5.6]{KPTT2} to say that upto $*$-isomorphism that fixes $\mS \ot_{\min} \mT$, $C^*_e(\mS \ot_{\min} \mT)= C^*_e(\mS) \ot_{C^*\text{-}\min} C^*_e(\mT)$.
\end{proof}


\section{Embedding of exact operator systems into $\mO_2$}\label{s:3}

The relationship between an operator system and its $C^*$-envelope is a mysterious one. In \cite{kirchberg}, Kirchberg and Wasserman gave an example of a universal separable exact operator system $\mS$ with  non exact $C^*$-envelope. Another interesting example was recently constructed by Lupini in \cite{lupini}, namely, the Gurarij operator system $\mathbb{GS}$, which is exact but does not admit any complete order embedding into any exact $C^*$-algebra. Thus in general, unlike $C^*$-algebras, separable exact operator systems need not embed into $\mO_2$. But, in the next theorem we prove an embedding theorem that shows that it is the exactness of the $C^*$-envelope, rather than that of the operator system, that makes an operator system embeddable into $\mO_2$.

\begin{theorem}\label{embed} 
For a separable operator system $\mS$ the $C^*$-envelope $C^*_e(\mS)$ is
exact if and only if there exist a unital complete order embedding of $\mS $ into $\mathcal{O}_2$.
\end{theorem}
\begin{proof}
 For the if part, let $\psi: \mS \rightarrow \mO_2$ be a complete order embedding. Then $\psi$ can be extended to a $*$-isomorphism on the $C^*$-envelope of $\mS$, say, $\hat{\psi} :C^*_e(\mS) \ra C^*_e(\psi(\mS))$ such that $\hat{\psi}|_{\mS}=\psi$,  using \Cref{extendiso}.
 Consider the $C^*$-algebra generated by $\psi(\mS) \subset \mO_2$,  $C^*(\psi(\mS)) \subseteq \mO_2$. Using \Cref{kirchexact}, $C^*(\psi(\mS))$ being a $C^*$-subalgebra of $\mO_2$, is exact. Further by universal (minimality) property of $C^*$-envelopes of operator systems, there exist a surjective $*$-homomorphism $\pi: C^*(\psi(\mS)) \ra C^*_e(\psi(\mS))$ such that the diagram commutes 
 \[\begin{tikzcd} \mO_2  &&\\
 C^*(\psi(\mS)) \arrow[hookrightarrow]{u}{} \arrow{drr}{\pi}   &   &  \\ 
\psi(\mS) \arrow[hookrightarrow]{u}{} \arrow{rr}{i_{\psi(\mS)}} && C^*_e(\psi(\mS))\\
\mS \arrow{u}{\psi} \arrow{rr}{i_{\mS}} & &  C^*_e(\mS) \arrow{u}[swap]{\hat{\psi}}
\end{tikzcd}\] 
where $i_{\mS}$ and $i_{\psi(\mS)}$ denote the natural complete order inclusion of $\mS$ and $\psi(\mS)$ into their respective $C^*$-envelopes. Thus $C^*_e(\psi(\mS)) $ is the $*$-homomorphic image of an exact $C^*$-algebra $C^*(\psi(\mS))$, and hence is exact (\Cref{propexact}(ii)). 
Therefore, $\hat{\psi}^{-1}(C^*_e(\psi(\mS))=C^*_e(\mS)$ is exact.
\\
Conversely, let $C^*_e(\mS)$ be exact.
Then, from Kirchberg's embedding theorem, there exist a complete order embedding $\phi$ of $ C^*_e(\mS)$ into  $\mO_2$. Then $\phi \circ i_{\mS} : \mS \ra \mO_{2}$ is the required unital complete order embedding of $\mS$ into $\mO_2$, where $i_{\mS}$ denote the natural complete order inclusion of $\mS$ into $C^*_e(\mS)$ .
\end{proof}

\begin{corollary}
 For an exact separable operator system $\mS$ containing enough unitaries of its $C^*$-envelope, $\mS$ embeds into $\mO_2$.  
\end{corollary}
\begin{proof}
By \cite[Proposition 10.12]{KPTT2}, for the case when $\mS$ contains enough unitaries of $C^*_e(\mS)$,  exactness of $\mS$ is equivalent to exactness of $C^*_e(\mS)$ and hence result follows by \Cref{embed}.
\end{proof} 

\begin{proposition}\label{corone} 
Let $\mT_1$ and $\mT_2$ be separable operator systems. If $C^*_e(\mT_1)$ and $C^*_e(\mT_2)$ are exact, then the operator system $\mT_1 \ot_{min} \mT_2$ embeds into $\mO_2$. Converse holds, if either both $C^*_e(\mT_1)$ and $C^*_e(\mT_2)$ are simple or both $\mT_1$ and $\mT_2$ contain enough unitaries of $C^*_e(\mT_1)$ and $C^*_e(\mT_2)$, respectively.
\end{proposition}
\begin{proof} Let $C^*_e(\mT_1)$ and $C^*_e(\mT_2)$ be exact $C^*$-algebras. Then using the Kirchberg's embedding theorem (\Cref{embed}) there exist complete order embeddings, $\phi_1 : C^*_e(\mT_1) \hookrightarrow \mO_2 $ and $\phi_2 : C^*_e(\mT_2) \hookrightarrow \mO_2.$ Since $C^*\text{-}\min$ is injective, we have the complete order isomorphism 
\begin{equation}\label{eqn1}
\phi_1  \ot_{\min} \phi_2 : C^*_e(\mT_1) \ot_{C^*\text{-}\min} C^*_e(\mT_2) \hookrightarrow \mO_2 \ot_{C^*\text{-}\min} \mO_2.
\end{equation}
Also, operator system $\min$ tensor product is injective \cite[Theorem 4.6]{KPTT1}, so that using the natural complete order inclusions $i_{\mT_1}$ and $i_{\mT_2}$ of $\mT_1$ and $\mT_2$ into their respective $C^*$-envelopes, we have the complete order isomorphism $i_{\mT_1} \ot i_{\mT_2}$ of  $\mT_1 \ot_{\min} \mT_2$ into $C^*_e(\mT_1) \ot_{\min} C^*_e(\mT_2)$. Further, since operator system $\min$ tensor product of $C^*$-algebras embeds complete order isomorphically into their $C^*\text{-}\min$ tensor product (\cite[Corollary 4.10]{KPTT1}), the complete order isomorphism can be considered as 
\begin{equation}\label{eqn2}
 i_{\mT_1} \ot_{\min} i_{\mT_2} : \mT_1 \ot_{\min} \mT_2 \hookrightarrow  C^*_e(\mT_1) \ot_{C^*\text{-}\min} C^*_e(\mT_2).
 \end{equation}
Using the isomorphism $\mO_2 \ot_{C^*\text{-}\min}  \mO_2 \cong \mO_2$ (\Cref{elliot}) and the composition of complete order isomorphisms in \Cref{eqn1,eqn2}, we have the required complete order isomrphism $$ \mT_1 \ot_{\min} \mT_2 \ \hookrightarrow    \mO_2.$$
Conversely, let there be an embedding of $ \mT_1 \ot_{\min}
\mT_2$ into $\mO_2$. In case, $C^*_e(\mT_i)$ is simple for $i=1,2$ or each $\mT_i$, $i=1,2$, contains enough unitaries of $C^*_e(\mT_i)$, respectively, then using \Cref{lemma}, $$C^*_e(\mT_1 \ot_{\min} \mT_2) \cong C^*_e(\mT_1) \ot_{C^*\text{-}\min} C^*_e(\mT_2),$$ which is separable (being the minimal $C^*$-tensor product of separable $C^*$-algebra). Thus \Cref{embed} implies that $C^*_e(\mT_1) \ot_{C^*\text{-}\min} C^*_e(\mT_2)$ is exact and hence, for each $i$, the $C^*$-subalgebras $C^*_e(\mT_i)$ (through the injective $*$-homomorphisms $C^*_e(\mT_1) \ni a_1 \mapsto a_1 \ot 1 \in C^*_e(\mT_1) \ot_{C^*\text{-}\min} C^*_e(\mT_2)$ and $C^*_e(\mT_2) \ni a_2 \mapsto 1 \ot a_2 \in C^*_e(\mT_1) \ot_{C^*\text{-}\min} C^*_e(\mT_2)$) is exact (\Cref{propexact}(iii)).
\end{proof}

Since $\min$ is associative (\cite[Theorem 4.6]{KPTT1}), preceding proposition can be extended to finite tensor product.

\begin{corollary}
Let $\mT_1,\mT_2,\cdots,\mT_m$; $m \in \mathbb{N}$ be separable operator systems. If, for each $i:1\leq i\leq m$, $C^*_e(\mT_i)$ is exact, then the operator system $ \mT_1 \ot_{\min} \mT_2 \ot_{\min} \cdots \ot_{\min} \mT_m$ embeds into $\mO_2$. Converse holds, if either $C^*_e(\mT_i)$ is simple for all $i$ or each $\mT_i$, $1 \leq i \leq m$, contains enough unitaries of $C^*_e(\mT_i)$, respectively.
\end{corollary}
\begin{proof}
Let $C^*_e(\mT_i)$ be exact for all $i=1,2,\cdots,m$. Then using the associativity of $\min$ and $C^*\text{-}\min$,
and the complete order isomorphism $\ot_{i=1}^{m} \mO_2 \cong \mO_2$ (\cite[Corollary 5.2.4]{rordam}) in the \Cref{corone}, we have the required complete order isomorphism as $$\mT_1 \ot_{\min} \mT_2 \ot_{\min} \cdots \ot_{\min} \mT_m \hookrightarrow \mO_2.$$
For the converse, the associativity of $\min$ and $C^*\text{-}\min$, extends the \Cref{lemma} to finite factors, so that if either $C^*_e(\mT_i)$ is simple for all $i$ or each $\mT_i$, $1 \leq i \leq m$, contains enough unitaries of $C^*_e(\mT_i)$, respectively, $$C^*_e(\mT_1 \ot_{\min} \mT_2 \ot_{\min} \cdots \ot_{\min} \mT_m ) \cong C^*_e(\mT_1) \ot_{C^*\text{-}\min} C^*_e(\mT_2) \ot_{C^*\text{-}\min} \cdots \ot_{C^*\text{-}\min} C^*_e(\mT_m).$$
So that the embedding $\mT_1 \ot_{\min} \mT_2 \ot_{\min} \cdots \ot_{\min} \mT_m$ into $\mO_2$, implies the exactness of $C^*_e(\mT_1 \ot_{\min} \mT_2 \ot_{\min} \cdots \ot_{\min} \mT_m ) $ (\Cref{embed}), and hence of each of its $C^*$-subalgebra $C^*_e(\mT_i);$ $i=1,2,\cdots,m$.
\end{proof}

Nuclearity properties of operator systems have been characterized in terms of various intrinsic properties (see \cite{kavnuc}), and their relation with the nuclearity of their $C^*$-envelope (see \cite{opsysnuc}) have been studied recently.  Recall, a generalization of the
notion of WEP was introduced, in \cite{KPTT2},  and was called \textit{the Double Commutant Expectation Property (DCEP)}. An operator system $\mS$ is
said to have the DCEP if for every complete order embedding $\mS
\subset B(H)$ there exists a completely positive map $\varphi : B(H)
\ra \mS''$ fixing $\mS$. A $C^*$-algebra thus has  DCEP if and and only it has WEP. In the next corollary, we give some nuclearity properties of operator system embeddable in $\mO_2$.

\begin{corollary}
For a separable operator system $\mS$ having an embedding in $\mO_2$, we have \begin{enumerate}
\item[(i)] $\mS$ is exact, and hence $(\min,\mathrm{el})$-nuclear.
\item[(ii)] $C^*_e(\mS)$ is nuclear if and only if $C^*_e(\mS)$ has the DCEP. In this case, $\mS$ is $(\min,\mathrm{ess})$-nuclear.
\item[(iii)] If $\mS$ has enough unitaries of $C^*_e(\mS)$, $\mS$ is $(\min,\mathrm{ess})$-nuclear if and only if $C^*_e(\mS)$ has the DCEP (or WEP).
\end{enumerate} 
\end{corollary}
\begin{proof}
\begin{enumerate}{}
\item[(i)] Since exactness passes to operator subsystems \cite[Corollary 5.8.]{KPTT2} and  $(\min,\mathrm{el})$-nuclearity is equivalent to exactness of operator system(\cite[Theorem 5.7]{KPTT2}, we have (i) from \Cref{embed}.
\item[(ii)] A unital $C^*$-algebra is nuclear if and only if it is exact and has DCEP (\cite[$\S 17$]{pisier} and \cite[$\S 7$]{KPTT2}; and nuclearity of $C^*$-envelope implies $(\min,\mathrm{ess})$-nuclearity of operator system (\cite[Proposition 4.2]{opsysnuc}), thus \Cref{embed} implies the result.
\item[(iii)] For an operator system having enough unitaries in $C^*_e(\mS)$, $(\min,\mathrm{ess})$-nuclearity is equivalent to nuclearity of $C^*_e(\mS)$ (\cite[Theorem 4.3]{opsysnuc}), therefore (iii) follows from (ii).
\end{enumerate}\end{proof}


\section{Tensor product with $\mS_2$ and $\mS_{\infty}$}\label{s:4}

Next we give a characterization of those operator systems, which get absorbed while considering the $C^*$-envelope of the operator system obtained by tensoring them finitely many times with $\mS_2$, in terms of their $C^*$-envelopes. 

\begin{proposition}\label{prop1} 
For a separable operator system $\mS$ with simple $C^*$-envelope, $C^*_e(\mS \ot_{\min=\mathrm{c}} \mS_2) \cong \mathcal{O}_2$ if and only if $C^*_e(\mS)$ is a nuclear $C^*$-algebra.
\end{proposition}
\begin{proof}
Since $C^*_e(\mS)$ is simple, using \Cref{lemma}, $$C^*_e(\mS) \ot_{C^*\text{-}\mathrm{min}} C^*_e(\mS_2) \cong C^*_e(\mS \ot_{\min} \mS_2) \cong \mathcal{O}_2,$$ and thus $C^*_e(\mS)$ is nuclear (\Cref{kirch2}).\\
Conversely, let $C^*_e(\mS)$ be a nuclear $C^*$-algebra. Using \Cref{lemma}, we have $$C^*_e(\mS \ot_{\mathrm{min}} \mS_2) \cong C^*_e(\mS) \ot_{C^*\text{-}\mathrm{min}} C^*_e(\mS_2),$$ and then by \Cref{kirch2} and \Cref{cntznvlp}, we have that $$C^*_e(\mS \ot_{\mathrm{min}} \mS_2) \cong C^*_e(\mS_2) \cong \mO_2.$$
\end{proof}

We have given the proof for the operator system of the form $\mS \ot_{\min} \mS_2$, but it can be generalized to $\mS \ot_{\min} \ot_{i=1}^{m} \mS_2$, using the identification $\ot_{i=1}^{m} \mO_2 =\mO_2$ (\cite[Corollary 5.2.4]{rordam})

\begin{corollary}\label{cor4}
 For any simple, unital, separable and nuclear $C^*$-algebra $A,$ we have $C^*_e(A \ot_{\min=\mathrm{max}} \mS_2 ) \cong \mathcal{O}_2$.
 \end{corollary}
\begin{proof}
 Follows directly using \Cref{prop1} and the fact that $C^*_e(A)=A$ (\cite[Proposition 2.3]{opsysnuc}).
 \end{proof}

Recall from \cite[Definition 1.1.15]{rordam}, for unital $C^*$-algebras $A$ and $B$, two completely positive maps $\phi, \psi: A \ra B$, are said to be \emph{unitarily equivalent} if there is a unitary $u$ in $B$ such that $u\psi(a)u^* = \phi(a)$ for all $a\in A$, in symbols $\phi \sim_u \psi$.
 If for every $\varepsilon > 0$ and for every finite subset $F$ of $A$ there is a unitary $u$ in $B$ with $\|u\psi(a)u^*-\phi(a)\|\leq \varepsilon$
for all $a \in  F$, then $\phi$ and $\psi$ are said to be \emph{approximately unitarily
equivalent}, denoted by $\phi \approx_u \psi$. 
Approximate unitary equivalence of completely positive maps has been used extensively in \cite[Theorem 5.1.1]{rordam} and \cite[Theorem 6.3.8]{rordam} to prove various isomorphisms of $C^*$-algebras involving $\mO_2$.

\begin{corollary}\label{cor2} 
For a separable operator system $\mS$ with simple $C^*$-envelope, the following are equivalent :
\begin{enumerate}
 \item[(i)] $C^*_e(\mS)$ is exact.
 \item[(ii)] $\mS \ot_{\min=\mathrm{c}} \mS_2$ embeds into $\mO_2$.
 \item[(iii)] $C^*_e(\mS \ot_{\min=\mathrm{c}} \mS_2)$ is  exact.
\item[(iv)] $C^*_e(\mS \ot_{\min=\mathrm{c}} \mS_2)$ can be embedded in $\mO_2$ as a $C^*$-subalgebra.
\end{enumerate}
Moreover, if any one of the above holds, then there exist injective $*$-homomorphisms, $\rho :\mO_2 \ra C^*_e(\mS) \ot_{C^*\text{-}\min} \mO_2\cong C^*_e(\mS \ot_{\min=\mathrm{c}} \mS_2)$ and  $\gamma: C^*_e(\mS \ot_{\min=\mathrm{c}} \mS_2) \ra \mO_2$ such that $\gamma \circ \rho \approx_u id_{\mO_2}$. And, in addition, if $\rho \circ \gamma \approx_u id_{C^*_e(\mS \ot_{\min=\mathrm{c}} \mS_2)}$, then $C^*_e(\mS)$ is nuclear and $C^*_e(\mS\ot_{\min=\mathrm{c}} \mS_2) \cong \mO_2$.
\end{corollary}
\begin{proof} 
Since, $C^*_e(\mS)$  and $C^*_e(\mS_2) = \mO_2$ are both simple and exact, by the converse of \Cref{corone} $\mS \ot_{\min} \mS_2 $ embeds into $\mO_2$. Also, if $\mS \ot_{\min} \mS_2$ embeds into $\mO_2$ then $C^*_e(\mS)$ is exact by \Cref{corone}. Thus, (i) and (ii) are equivalent.
\\
\Cref{embed} implies (ii) $\Longleftrightarrow$ (iii).
\\
Equivalence of (iii) and (iv) follows using Kirchberg's exact embedding theorem (\Cref{kirchexact}).
\\
Now, suppose $\mS$ satisfies any of the above. Let $\rho: \mO_2 \hookrightarrow C^*_e(\mS) \ot_{C^*\text{-}\min} C^*_e(\mS_2) = C^*_e(\mS \ot_{\min=\mathrm{c}} \mS_2)$; $a \mapsto 1 \ot a$ and $\gamma: C^*_e(\mS \ot_{\min=\mathrm{c}} \mS_2) \hookrightarrow \mO_2 $ be the injective $*$-homomorphism.
Using \cite[Theorem 5.1.1]{rordam}, we know that any injective $*$-homomorphism from $\mO_2$ into $\mO_2$, is approximately unitarily equivalent to $id_{\mO_2}$. So that $\gamma \circ \rho \approx_{u}  id_{\mO_2}$.
\\
In case we further have  $\rho \circ \gamma \approx_{u} id_{C^*_e(\mS \ot_{\min=\mathrm{c}} \mS_2)}$, then using \cite[Theorem 6.3.8]{rordam}(ii), $C^*_e(\mS \ot_{\min=\mathrm{c}} \mS_2)$ is isomorphic to $\mO_2$ and hence by \Cref{prop1} $C^*_e(\mS)$ is nuclear.
\end{proof}

\begin{remarks}
\emph{In case the complete order embedding $\mS \ot_{\min=\mathrm{c}} \mS_2$ into $\mO_2$ obtained in (ii) of the above corollary is such that $\mO_2$ is a $C^*$-cover, then trivially $C^*_e(\mS \ot_{\min=\mathrm{c}} \mS_2) = \mO_2$ (\Cref{simplecover}).}
\end{remarks}


Recall from \cite{rordam}, a simple $C^*$-algebra $A$ is said to be\textit{ purely infinite} if  $A$ is not isomorphic to $\mathbb{C}$ and for every pair of non-zero elements $a$ and $b$ in $A$ there exists $x$ in $A$ such that $b=x^*ax$. In fact there are six equivalent conditions that are used to define a unital and simple $C^*$-algebra to be purely infinite(\cite[Proposition 4.1.1]{rordam}.
\\\\
We now characterize those operator system whose $C^*$-envelopes remain unaffected by tensoring finitely many times with $\mS_{\infty}$. 
\\
Proof is on the same lines as that of \Cref{prop1} and uses the Kirchberg's characterization of simple, purely infinite, nuclear $C^*$-algebras (\Cref{Oinfinity}). 
 
\begin{proposition}\label{prop2} 
Let $\mS$ be a separable operator system with simple $C^*$-envelope $C^*_e(\mS).$ Then $C^*_e(\mS)$ is a  nuclear and purely infinite $C^*$-algebra if and only if $C^*_e(\mS \ot_{\min=\mathrm{c}} \mS_{\infty}) \cong C^*_e(\mS)$.
\end{proposition}

Again we stated the last Proposition for operator system of the form $\mS \ot_{\min=\mathrm{c}} \mS_{\infty}$ but it can be generalized to operator system of the form $\mS \ot_{\min=\mathrm{c}} \ot_{i=1}^m \mS_{\infty}$, using the identification $\mO_{\infty} =\ot_{i=1}^{m} \mO_{\infty}$ (\cite[Theorem 7.2.6]{rordam}).
 
\begin{corollary} 
For a separable operator system $\mS$ with simple, nuclear and purely infinite $C^*$-envelope, there exist a complete order embedding of $\mS \ot_{\min=\mathrm{c}} \mS_{\infty} $ into $C^*_e(\mS)$.
\end{corollary}

Using the fact that for a unital $C^*$-algebra $C^*_e(A)=A$ (\cite[Proposition 2.3]{opsysnuc}), we have:

\begin{corollary}\label{cor3} 
For any unital, simple, nuclear, separable and purely infinite $C^*$-algebra $A$, $C^*_e(A \ot_{\min=\mathrm{max}} \mS_{\infty}) \cong A.$ 
\end{corollary}

\begin{corollary} 
For a separable operator system $\mS$ with simple, nuclear $C^*$-envelope $C^*_e(\mS)$, if $\mS \cong \mS \ot_{\mathrm{min}} \mS_{\infty} $, then $\mS$ is infinite dimensional and $C^*_e(\mS)$ is purely infinite.
\end{corollary}
\begin{proof}
If $\mS \ot_{\mathrm{min}} \mS_{\infty} \cong \mS$, proof follows from \Cref{extendiso} and \Cref{prop2}. 
\end{proof}

\begin{remarks}
\emph{Converse of the above corollary is not known. Note that $\mS_n$ is a finite dimensional operator system with purely infinite and simple $C^*_e(\mS_n)=\mathcal{O}_n$(\Cref{cuntz}). But 
 $\mS_n \ncong \mS_n \ot_{\min} \mS_{\infty}$.} 
 \end{remarks}


\section{Applications}\label{s:5} 
 
The results proved in this article can be applied to some recently introduced operator systems with known $C^*$-envelopes to check their embeddability in $\mO_2$, and to describe the $C^*$-envelopes of some operator systems obtained after tensoring them with $\mS_2$ or $\mS_{\infty}$.
\\
Recall from \cite{disgrp}, an operator system $\mS(\mF)$ was associated to $C^*(G)$, the full group $C^*$-algebra
of the group $G$ for countable discrete group $G$,
$\mF$ being a generating set of $G$ as; $\mS(\mF) := \text{span}\{1, u, u^*:\ u \in \mF\}
\subset C^*(G)$; it was shown in \cite[Proposition 2.2]{disgrp} that $C^*_e(\mS(\mF))=C^*(G)$. On the similar lines, another natural operator system was associated to reduced group $C^*$-algebra in \cite{opsysnuc}, namely,
$\mS_r(\mF) := \text{span}\{1, u, u^*:\ u \in \mF\} \subset
C^*_r(G)$. Further, $C^*_e(\mS_r(\mF))=C^*_r(G)$ (see \cite[Proposition 2.9]{opsysnuc}).
\\
Kavruk et al.~in
\cite{KPTT1} associated an operator system to a finite graph $G$ with $n$-vertices, $\mS_G$ as the finite
dimensional operator subsystem of $M_n(\C)$ given by
 $\mS_G=\mathrm{span} \{\{E_{i,j}: (i,j) \in G\}\cup\{E_{i,
  i}: 1 \leq i \leq n\} \} \subseteq M_n(\C),$ where $\{E_{i,j}\}$ is
the standard system of matrix units in $M_n(\C)$ and $(i, j)$ denotes
(an unordered) edge in $G$. From the proof of \cite[Theorem 3.2]{ortiz}, we now know that, for a connected graph $G$ on $n$-vertices, $C^*_e(\mS_G) =M_n$.

\begin{example}
\emph{As an application of \Cref{embed}, the following operator systems embed into $\mO_2$}:
\begin{enumerate}
\item[(i)] The operator system $\mS(\mathfrak{u}) \subseteq C^*(G)$; where $G$ is a finitely generated discrete amenable group.
\item[(ii)] $\mS_r(\mathfrak{u}) \subseteq C^*_r(G)$; where $G$ is any exact discrete group. In particular for $G=F_n$, the free group on $n$-generators, $\mS_r(\mathfrak{u}_n) \subset C^*_r(F_n)$ embeds into $\mO_2$.
\item[(iii)] $\mS_G \subset M_n$; where $G$ is a connected graph on $n$-vertices embed into $\mO_2$.
\end{enumerate}
\emph{On the other hand, $\mathcal{S}(\mathfrak{u}_n) \subseteq C^*(F_n)$ does not embed into $\mO_2$.}
\end{example}

The following applications are immediate from \Cref{cuntz}, \Cref{cuntzopsysnuc} and \Cref{cor4}.

\begin{example}
\emph{$C^*_e(\mS_n \ot_{\min=\mathrm{c}} \mS_2) \cong \mathcal{O}_2 $ for all $2 \leq n \leq \infty$.}
\end{example}

Similarly, as a direct application of \Cref{prop2}, \Cref{cntznvlp} and \Cref{cuntz}, we have:

\begin{example}
\emph{$C^*_e(\mS_n \ot_{\min=\mathrm{c}} \mS_{\infty} ) \cong \mO_n$ for all $2 \leq n \leq \infty$.}
\end{example}

\begin{example}
\emph{$C^*_e(M_n \ot_{\min=\mathrm{c}} \mS_2) \cong \mathcal{O}_2$ for all $n \in \mathbb{N}$.}
\end{example}

We know that $C^*(G)$ is never simple unless $G=\mathbb{C}$; as it always has a one dimensional quotient coming from the trivial representation of $G$; and has an ideal of co-dimension $1$, called \emph{the augmented ideal}. But for $n \geq 2$; $C^*_r(F_n)$(the reduced group algebra of free group with $n$-generator) is always simple. 

\begin{example} 
\emph{Consider $\mS_r(\mathfrak{u}_n) \subseteq C^*_r(F_n)$ for $n \geq 2$, then $C^*_e(\mS_r(\mathfrak{u}_n)) =C^*_r(F_n)$, which is simple, separable, unital and exact but not nuclear, then $C^*_e(\mS_r(\mathfrak{u}_n) \ot_{\min} \mathcal{S}_n) \cong C^*_r(F_n) \ot_{\min} \mathcal{O}_2$ is a proper $C^*$-subalgebra of $\mathcal{O}_2$.}
\end{example}

\begin{example}
\emph{For a connected graph $G$ on $n$-vertices, $C^*_e(\mS_{G} \ot_{\min} \mS_2)  \cong M_n \ot_{C^*\text{-}\min} \mathcal{O}_2 \cong \mathcal{O}_2$; where $\mathcal{S}_G$ is the graph operator system.}
\end{example}

Argerami and
Farenick \cite{argerami1,argerami2} defined operator systems generated by a single bounded linear operator $T$ acting
on a complex Hilbert space $\mH$ as the unital
self-adjoint subspace $\mO\mS(T) = \text{span} \{ 1, T, T^*\} \subset
\B.$

\begin{example}
\emph{Recall from \cite[Proposition 3.2]{argerami1}, for $\C^\times :=\C \diagdown \{0\}$ and $\xi= (\xi_1, \xi_2, \ldots,
\xi_d) \in (\C^\times)^d$, the irreducible weighted
  unilateral shift with weights $\xi_1, \xi_2, \ldots, \xi_d $ is the
operator $W(\xi)$ on $\C^{d+1}$ given by the
matrix 
$$W(\xi)=\begin{bmatrix} 
0 & & & & 0\\ 
\xi_1 & 0 & & & \\ 
& \xi_2 & \ddots & & \\ 
& & \ddots & 0 & \\ 
& & & \xi_d & 0
\end{bmatrix},$$ 
and
$C^*_e(\mO\mS({W(\xi)}))=M_{d+1}(\C)$.
\\
Therefore, $\mO\mS({W(\xi)})$ and $ \mO\mS({W(\xi)}) \ot_{\mathrm{min}} \mS_2$ embed into $\mO_2$ using \Cref{cor2} and $C^*_e(\mO\mS({W(\xi)}) \ot_{\min} \mS_2)\cong \mO_2$ using \Cref{prop1}.}
\end{example}

An operator $J$ on an $n$-dimensional Hilbert space $\mH$
is a basic Jordan block if there is an orthonormal basis of $\mH$ for
which $J$ has a matrix representation of the form $$J_n(\lambda)
:= \begin{bmatrix} \lambda & 1 & 0 & \ldots & 0 \\ 0 & \lambda & 1 &
  \ddots & \vdots \\ \vdots & \ddots & \ddots & \ddots & 0\\ \vdots &
  & \ddots & \ddots & 1 \\ 0 & \ldots & \ldots & 0 & \lambda
\end{bmatrix}$$
for some $\lambda \in \C.$

\begin{example}
\emph{By \cite[Proposition 2.2]{argerami2}, for $J=\bigoplus_{k=1}^{\infty} J_{m_k}(\lambda) \in B(l^2(\N))$ with $m
:= \sup\{m_k: k \in \N\} <\infty $, $$C^*_e(\mO\mS(J))= M_m(\C).$$
Thus, $\mO\mS(J)$ and $\mO\mS(J) \ot_{\min=\mathrm{c}} \mS_2$ embed into $\mO_2$ and $C^*_e(\mO\mS(J) \ot_{\min=\mathrm{c}} \mS_2) \cong \mO_2$ using \Cref{prop1} and \Cref{cor2}.}
\end{example}

\begin{example}
\emph{If $J=\bigoplus_{k=1}^{n}( J_{m_k}({\lambda_{k}}) \ot I_{d_k})$, with
$\lambda_1 > \lambda_2 > \cdots > \lambda_n$ all real with $\max\{m_2,
\ldots, m_{n-1}\} \leq \min\{m_1, m_n\}$, then, by \cite[Corollary 2.12]{argerami2}},
  \\
\emph{$C^*_e(\mO\mS(J))$ is a nuclear, simple, separable $C^*$-algebra, for the cases $ \; m_1 = 1, m_n \geq 2, |\lambda_1 -\lambda_n| \leq \cos \frac{\pi}{(m_n+1)}$ and $ m_1 \geq 2, m_n = 1 , |\lambda_1 -\lambda_n| \leq \cos \frac{\pi}{(m_1+1)}.$}\\ \emph{Therefore, for these cases $\mO\mS(J)$ and $\mO\mS(J) \ot_{\min} \mS_2$ embed into $\mO_2$ and $C^*_e(\mO\mS(J) \ot_{\min} \mS_2) \cong \mO_2.$}
\end{example}

\begin{example} 
\emph{Using \Cref{cor3} and \Cref{cuntz}; $C^*_e(\mO_{n} \ot_{\min=\mathrm{max}} \mS_{\infty}) \cong \mO_n$ for all $2 \leq n \leq \infty.$}
\end{example}

\subsection*{Acknowledgements}
{\small The authors are grateful to Ved Prakash Gupta for his careful reading of the manuscript and suggestions which led to many improvements. The authors would also like to
  thank the referee for valuable comments and feedback.}

\bibliographystyle{plain}
\bibliography{Luthra_kumar_Emb_C_Env_Exact_Op_Sys_Ref}

\end{document}